\documentclass[11pt]{article}
\usepackage{amsfonts}
\usepackage{color}
\usepackage{graphicx}
\usepackage{amssymb}
\usepackage{amsmath}

\newtheorem{theorem}{Theorem}

\newtheorem{definition}[theorem]{Definition}

\newtheorem{remark}[theorem]{Remark}

\newenvironment{proof}[1][Proof]{\noindent\textbf{#1.} }{\ \rule{0.5em}{0.5em}}

\begin{document}

\title{Heteroclinic and homoclinic solutions for nonlinear second-order
coupled systems with phi-Laplacians}
\author{Robert de Sousa$^{(\ddag )}$ and Feliz Minh\'{o}s$^{(\dag )}$ \\
$^{(\ddag )}${\small Faculdade de Ci\^{e}ncias e Tecnologia,}\\
{\small N$\acute{u}$cleo de Matem\'{a}tica e Aplica\c{c}\~{o}es (NUMAT),}\\
{\small Universidade de Cabo Verde. Campus de Palmarejo,}\\
{\small 279 Praia, Cabo Verde}\\
{\small robert.sousa@docente.unicv.edu.cv}\\
$^{(\dag )}${\small Departamento de Matem\'{a}tica, Escola de Ci\^{e}ncias e
Tecnologia,}\\
{\small Centro de Investiga\c{c}\~{a}o em Matem\'{a}tica e Aplica\c{c}\~{o}%
es (CIMA),}\\
{\small \ Instituto de Investiga\c{c}\~{a}o e Forma\c{c}\~{a}o Avan\c{c}ada,
}\\
{\small \ Universidade de \'{E}vora. Rua Rom\~{a}o Ramalho, 59, }\\
{\small \ 7000-671 \'{E}vora, Portugal}\\
{\small fminhos@uevora.pt}\\
}
\date{}
\maketitle

\begin{abstract}
In this paper we present sufficient conditions for the existence of
heteroclinic or homoclinic solutions for second order coupled systems of
differential equations on the real line.

We point out that it is required only conditions on the homeomorphisms and
no growth or asymptotic conditions are assumed on the nonlinearities.

The arguments make use of the fixed point theory, $L^{1}$-Carath\'{e}odory
functions and Schauder's fixed point theorem.

An application to a family of second order nonlinear coupled systems of two
degrees of freedom, shows the applicability of the main theorem.\bigskip

\bigskip

\end{abstract}

\section{Introduction}

In this paper, we consider the second order coupled system on the real line
\begin{equation}
\left\{
\begin{array}{l}
\left( a(t)\phi \big(u^{\prime }(t)\big)\right) ^{\prime
}=f(t,u(t),v(t),u^{\prime }(t),v^{\prime }(t)),\text{ } \\
\\
\left( b(t)\psi \big(v^{\prime }(t)\big)\right) ^{\prime
}=h(t,u(t),v(t),u^{\prime }(t),v^{\prime }(t)),\text{ }t\in \mathbb{R},%
\end{array}%
\right.   \label{art8}
\end{equation}%
with $\phi $ and $\psi $ increasing homeomorphisms verifying some adequate
relations on their inverses, $a,b:\mathbb{R}\rightarrow (0,\,+\infty \lbrack
$ are continuous functions, $f,h:\mathbb{R}^{5}\rightarrow \mathbb{R}$ are $%
L^{1}$-Carath\'{e}odory functions, together with asymptotic conditions
\begin{equation}
u(-\infty )=A,\;u^{\prime }(+\infty )=0,\;\;v(-\infty )=B,\;\;v^{\prime
}(+\infty )=0  \label{cond8.1}
\end{equation}%
for $A,B\in \mathbb{R}.$

Heteroclinic trajectories play an important role in geometrical analysis of
dynamical systems, connecting unstable and stable equilibria having two or
more equilibrium points, \cite{Izhikevich}. In fact, the homoclinic or
heteroclinic orbits are a kind of spiral structures, which are general
phenomena in nature, \cite{Tao+Jun}. Graphical illustrations and a very
complete explanation on homoclinics and heteroclinics bifurcations can be
seen in \cite{Homburg+Sandstede}. A planar homoclinic theorem and
heteroclinic orbits, to analyze fluid models, is studied in \cite{Bertozzi}.
Applications of dynamic systems techniques to the problem of heteroclinic
connections and resonance transitions, are treated in \cite{W+M+J+S}, on
planar circular domains. To prove the existence of heteroclinic solutions,
for a class of non-autonomous second-order equations, see \cite{C.Alves,
Ellero+Zanolin, Marcelli+Papalini}. Topological, variational and
minimization methods to find heteroclinic connections can be found in \cite%
{Zelati+ Rabinowitz}.

On heteroclinic coupled systems, among many published works, we highlight
some of them:

In \cite{Aguiar+Ashwin}, Aguiar et al. consider the dynamics of small
networks of coupled cells, with one of the points, analyzed as invariant
subsets, can support robust heteroclinic attractors;

In \cite{Ashwin+Karabacak}, Ashwin and Karabacak study coupled phase
oscillators and discuss heteroclinic cycles and networks between partially
synchronized states and in \cite{Karabacak+Ashwin}, they analyze coupled
phase oscillators, highlighting a dynamic mechanism, nothing more than a
heteroclinic network;

In \cite{Ashwin+Orosz}, the authors investigate such heteroclinic network
between partially synchronized states, where the phases cluster are divided
into three groups;

Moreover, in \cite{Feng+Hu}, the authors present some applications, results,
methods and problems that have been recently reported and, in addition, they
suggest some possible research directions, and some problems for further
studies on homoclinics and heterocli\-nics.

Cabada and Cid, in \cite{Cabada+Cid}, study the following boundary value
problem on the real line
\begin{equation*}
\left\{
\begin{array}{l}
\left( \phi \big(u^{\prime }(t)\big)\right) ^{\prime }=f(t,u(t),u^{\prime
}(t)),\text{ on }\mathbb{R}, \\
u(-\infty )=-1,\;\;\;u(+\infty )=1,%
\end{array}%
\right.
\end{equation*}%
with a singular $\phi $-Laplacian operator where $f$ is a continuous
function that satisfies suitable symmetric conditions.

In \cite{Calamai}, Calamai discusses the solvability of the following
strongly nonli\-near problem:
\begin{equation*}
\left\{
\begin{array}{l}
\left( a(x(t))\phi \big(x^{\prime }(t)\big)\right) ^{\prime
}=f(t,x(t),x^{\prime }(t)),\;\;\;t\in \mathbb{R}, \\
x(-\infty )=\alpha ,\;\;\;x(+\infty )=\beta ,%
\end{array}%
\right.
\end{equation*}%
where $\alpha <\beta $, $\phi :(-r,\,r)\rightarrow \mathbb{R}$ is a general
increasing homeomorphism with bounded domain (singular $\phi $-Laplacian), $%
a $ is a positive, continuous function and $f$ is a Carath\'{e}odory
nonlinear function.

Recently, in \cite{Kajiwara}, Kajiwara proved the existence of a
heteroclinic solution of the FitzHugh-Nagumo type reaction-diffusion system,
under certain conditions on the heterogeneity.

Motivated by these works and applying the techniques suggested in \cite%
{YLiu+SChen, F+HC, F+heteroclinic, FM+Robert}, we apply the fixed point
theory, to obtain sufficient conditions for the existence of heteroclinic
solutions of the coupled system (\ref{art8}), (\ref{cond8.1}), assuming some
adequate conditions on $\phi ^{-1}$, $\psi ^{-1}$.

We emphasize that it is the first time where heteroclinic solutions for
second order coupled differential systems are considered for systems with
full nonlinearities depending on both unknown functions and their first
derivatives. An example illustrates the potentialities of our main result,
and an application to coupled nonlinear systems of two degrees of freedom
(2-DOF), shows the applicability of the main theorem.

This paper is organized as it follows: Section 2 contains some preliminary
results. In section 3 we present the main theorem: an existence result of,
at least, a pair of heteroclinic solutions. An application to a family of
coupled 2-DOF nonlinear systems is presented in last section.

\section{Notations and preliminary results}

Consider the following spaces
\begin{equation*}
X:=\left\{ x\in C^{1}\left( \mathbb{R}\right) :\,\underset{|t|\rightarrow
\infty }{\lim }x^{(i)}(t)\in \mathbb{R},\, i=0,1\right\},
\end{equation*}
equipped with the norm%
\begin{equation*}
\left\Vert x\right\Vert _{X}=\max \left\{ \left\Vert x\right\Vert _{\infty
},\left\Vert x^{\prime }\right\Vert _{\infty }\right\},
\end{equation*}%
where
\begin{equation*}
\left\Vert x\right\Vert _{\infty }:=\sup_{t\in \mathbb{R}}|x(t)|,
\end{equation*}%
and $X^2:=X\times X$ with%
\begin{equation*}
\left\Vert \left( u,v\right) \right\Vert _{X^{2}}=\max \left\{ \left\Vert
u\right\Vert _{X},\left\Vert v\right\Vert _{X}\right\} .
\end{equation*}

It can be proved that $(X,\left\Vert \cdot \right\Vert _{X})$ and $%
(X^2,\left\Vert \cdot \right\Vert _{X^{2}})$ are Banach spaces.

\begin{remark}
If $w \in X$ then $w^{\prime}(\pm \infty)=0.$  \label{Remkc}
\end{remark}

By solution of problem (\ref{art8}), (\ref{cond8.1}) we mean a pair $%
(u,v)\in X^{2}$ such that
\begin{equation*}
a(t)\phi (u^{\prime }(t))\in W^{1,1}(\mathbb{R})\text{ and }b(t)\psi
(v^{\prime }(t))\in W^{1,1}(\mathbb{R}),
\end{equation*}
verifying (\ref{art8}), (\ref{cond8.1}).

For the reader's convenience we consider the definition of $L^{1}-$ Carath%
\'{e}odory functions:

\begin{definition}
\label{Lcar7} A function $g:\mathbb{R} ^{5}\rightarrow \mathbb{R}$ is $%
L^{1}- $ Carath\'{e}odory if

\begin{enumerate}
\item[$i)$] for each $(x,y,z,w)\in \mathbb{R}^4$, $t\mapsto g(t,x,y,z,w)$ is
measurable on $\mathbb{R};$

\item[$ii)$] for a.e. $t\in \mathbb{R},$ $(x,y,z,w)\mapsto g(t,x,y,z,w)$ is
continuous on $\mathbb{R}^{4};$

\item[$iii)$] for each $\rho >0$, there exists a positive function $%
\vartheta _{\rho }\in L^{1}\left(\mathbb{R}\right)$ such that, whenever $%
x,y,z,w \in [-\rho,\,\rho]$, then

\begin{equation}
\left\vert g(t,x,y,z,w)\right\vert \leq \vartheta _{\rho }(t),\text{ }%
a.e.\;\,\text{{}}t\in \mathbb{R}.  \label{caract8}
\end{equation}
\end{enumerate}
\end{definition}

Along this chapter we assume that

\begin{itemize}
\item[(H1)] $\phi $,$\psi :\mathbb{R}\longrightarrow \mathbb{R}$ are
increasing homeomorphisms such that

\indent\textbf{a)} $\phi (\mathbb{R})=\mathbb{R},\;\;\,\phi (0)=0,\;\;\,\psi
(\mathbb{R})=\mathbb{R},\;\;\,\psi (0)=0;$

\indent\textbf{b)} $|\phi ^{-1}(x)|\leq \phi ^{-1}(|x|),$\;\;\, $|\psi
^{-1}(x)|\leq \psi ^{-1}(|x|).$

\item[(H2)] $a,b:\mathbb{R}\rightarrow (0,\,+\infty[$ are positive
continuous functions such that

\begin{equation*}
\underset{t\rightarrow \pm \infty }{\lim }\frac{1}{a(t)}\in \mathbb{R}\text{
and }\underset{t\rightarrow \pm \infty }{\lim }\frac{1}{b(t)}\in \mathbb{R}.
\end{equation*}
\end{itemize}

A convenient criterion for the compacity of the operators is given by next
theorem:

\begin{theorem}
\label{F+H}(\cite{F+HC}, Theorem 2.3) A set $M\subset X$ is relatively
compact if the following conditions hold:

\begin{enumerate}
\item[$i)$] both $\{t\rightarrow x(t):x\in M\}$ and $\{t\rightarrow
x^{\prime }(t):x\in M\}$ are uniformly bounded;

\item[$ii)$] both $\{t\rightarrow x(t):x\in M\}$ and $\{t\rightarrow
x^{\prime }(t):x\in M\}$ are equicontinuous on any compact interval of $%
\mathbb{R}$;

\item[$iii)$] both $\{t\rightarrow x(t):x\in M\}$ and $\{t\rightarrow
x^{\prime }(t):x\in M\}$ are equiconvergent at $\pm \infty $, that is, for
any given $\epsilon >0$, there exists $t_{\epsilon }>0$ such that
\begin{equation*}
\left\vert f(t)-f(\pm \infty )\right\vert <\epsilon ,\;\left\vert f^{\prime
}(t)-f^{\prime }(\pm \infty )\right\vert <\epsilon ,\forall |t|>t_{\epsilon
},f\in M.
\end{equation*}
\end{enumerate}
\end{theorem}

The existence tool will be given by Schauder's fixed point theorem:

\begin{theorem}
\label{schauder} (\cite{zeidler}) Let $Y$ be a nonempty, closed, bounded and
convex subset of a Banach space $X$, and suppose that $P:Y\rightarrow Y$ is
a compact operator. Then $P$ has at least one fixed point in $Y$.
\end{theorem}

\section{Existence of heteroclinics}

In this section we prove the existence for a pair of heteroclinic solutions
to the coupled system (\ref{art8}), (\ref{cond8.1}), for some constants $%
A,B\in \mathbb{R}.$

\begin{theorem}
\label{main8} Let $\phi ,\,\psi :\mathbb{R}\rightarrow \mathbb{R}$ be
increasing homeomorphisms and $a,b:\mathbb{R}\rightarrow (0,\,+\infty
\lbrack $ continuous functions satisfying (H1) and (H2). Assume that $f,h:%
\mathbb{R}^{5}\rightarrow \mathbb{R}$ are $L^{1}-$Carath\'{e}odory functions
and there is $R>0$ and $\vartheta _{R},\;\theta _{R}\in L^{1}\left( \mathbb{R%
}\right) $ such that
\begin{equation}
\int_{-\infty }^{+\infty }\phi ^{-1}\left( \frac{\int_{-\infty }^{+\infty
}\vartheta _{R}(r)dr}{a(s)}\right) ds<+\infty ,  \label{majmiu1}
\end{equation}%
\begin{equation}
\int_{-\infty }^{+\infty }\psi ^{-1}\left( \frac{\int_{-\infty }^{+\infty
}\theta _{R}(r)dr}{b(s)}\right) ds<+\infty ,  \label{majmiu2}
\end{equation}%
with
\begin{equation*}
\sup_{t\in \mathbb{R}}\phi ^{-1}\left( \frac{\int_{-\infty }^{+\infty
}\vartheta _{R}(r)dr}{a(t)}\right) ds<+\infty ,\;\sup_{t\in \mathbb{R}}\psi
^{-1}\left( \frac{\int_{-\infty }^{+\infty }\theta _{R}(r)dr}{b(t)}\right)
ds<+\infty ,
\end{equation*}%
\begin{equation}
|f(t,x,y,z,w)|\leq \vartheta _{R}(t),  \label{majmiuL1}
\end{equation}%
\begin{equation}
|h(t,x,y,z,w)|\leq \theta _{R}(t),  \label{majmiuL2}
\end{equation}%
whenever $x,y,z,w\in \lbrack -R,\,R].$\newline
Then for given $A,B\in \mathbb{R}$ the problem (\ref{art8}), (\ref{cond8.1})
has, at least, a pair of heteroclinic solutions $(u,v)\in X^{2}$.
\end{theorem}

\begin{proof}
Define the operators $T_{1}:X^{2}\,\rightarrow X$, $T_{2}:X^{2}\,\rightarrow
X$ and $T\,:\,X^{2}\rightarrow X^{2}$ by
\begin{equation}
T\left( u,v\right) =\left( T_{1}\left( u,v\right) ,T_{2}\left( u,v\right)
\right) ,  \label{T8}
\end{equation}%
with
\begin{eqnarray*}
\left( T_{1}\left( u,v\right) \right) \left( t\right)  &=&\int_{-\infty
}^{t}\phi ^{-1}\left( \frac{\int_{-\infty }^{s}f(r,u(r),v(r),u^{\prime
}(r),v^{\prime }(r))dr}{a(s)}\right) ds+A, \\
\left( T_{2}\left( u,v\right) \right) \left( t\right)  &=&\int_{-\infty
}^{t}\psi ^{-1}\left( \frac{\int_{-\infty }^{s}h(r,u(r),v(r),u^{\prime
}(r),v^{\prime }(r))dr}{b(s)}\right) ds+B,
\end{eqnarray*}%
with $A$ and $B$ given by (\ref{cond8.1}).

In order to apply Theorem \ref{schauder}, we shall prove that $T$ is compact
and has a fixed point.

To simplify the proof, we detail the arguments only for $T_{1}\left(
u,v\right)$, as for the operator $T_{2}\left( u,v\right) $ the technique is
similar.

To be clear, we divide the proof into claims \textbf{(i)-(v)}.\medskip

\textbf{(i)} $T$\textit{\ is well defined and continuous in }$X^2$.\medskip

Let $(u,v)\in X^{2}$ and take $\rho >0$ such that $\left\Vert \left(
u,v\right) \right\Vert _{X^{2}}<\rho $. As $f$ is a $L^{1}-$ Carath\'{e}%
odory function, there exists a positive function $\vartheta _{\rho }\in
L^{1}\left( \mathbb{R}\right) $ verifying (\ref{majmiuL1}). So,%
\begin{eqnarray*}
&&\int_{-\infty }^{t}|f(r,u(r),v(r),u^{\prime }(r),v^{\prime }(r))|dr \\
&\leq &\int_{-\infty }^{+\infty }|f(r,u(r),v(r),u^{\prime }(r),v^{\prime
}(r))|dr\leq \int_{-\infty }^{+\infty }\vartheta _{\rho }(t)dt<+\infty .
\end{eqnarray*}

So, $T_{1}$ is continuous on $X$. Furthermore,
\begin{equation*}
\left( T_{1}\left( u,v\right) \right) ^{\prime }(t)=\phi ^{-1}\left( \frac{%
\int_{-\infty }^{t}f(r,u(r),v(r),u^{\prime }(r),v^{\prime }(r))dr}{a(t)}%
\right)
\end{equation*}%
is also continuous on $X$ and, therefore, $T_{1}\left( u,v\right) \in C^{1}(%
\mathbb{R})$.

By (\ref{cond8.1}), (\ref{majmiu1}), (\ref{majmiuL1}) and (H2),
\begin{eqnarray*}
&&\lim_{t\rightarrow -\infty }T_{1}\left( u,v\right) (t) \\
&=&\lim_{t\rightarrow -\infty }\int_{-\infty }^{t}\phi ^{-1}{\left( \frac{%
\int_{-\infty }^{s}f(r,u(r),v(r),u^{\prime }(r),v^{\prime }(r))dr}{a(s)}%
\right) }ds+A=A,
\end{eqnarray*}%
\begin{eqnarray*}
&&\lim_{t\rightarrow +\infty }T_{1}\left( u,v\right) (t) \\
&=&\int_{-\infty }^{+\infty }\phi ^{-1}{\left( \frac{\int_{-\infty
}^{s}f(r,u(r),v(r),u^{\prime }(r),v^{\prime }(r))dr}{a(s)}\right) }%
ds+A<+\infty ,
\end{eqnarray*}%
and%
\begin{eqnarray*}
\lim_{t\rightarrow \pm \infty }\left( T_{1}\left( u,v\right) (t)\right)
^{\prime } &=&\lim_{t\rightarrow \pm \infty }\phi ^{-1}{\left( \frac{%
\int_{-\infty }^{t}f(r,u(r),v(r),u^{\prime }(r),v^{\prime }(r))dr}{a(t)}%
\right) } \\
&\leq &\underset{t\rightarrow \pm \infty }{\lim }\phi ^{-1}\left( \frac{%
\int_{-\infty }^{+\infty }\vartheta _{\rho }(r)dr}{a(t)}\right) <+\infty .
\end{eqnarray*}

Therefore, $T_{1}\left( u,v\right) \in X$, and, by the same arguments, $%
T_{2}\left( u,v\right) \in X$. So, $T\left( u,v\right) \in X^2$.\medskip

\textbf{(ii)} $TM$\textit{\ is uniformly bounded on }$M\subseteq X^2$,
\textit{\ for some bounded $M$}.\medskip

Let $M$ be a bounded set of $X^2$, defined by
\begin{equation}
M:=\{(u,v)\in X^2:\max \left\{ \left\Vert u\right\Vert _{\infty },\left\Vert
u^{\prime }\right\Vert _{\infty },\left\Vert v\right\Vert _{\infty
},\left\Vert v^{\prime }\right\Vert _{\infty }\right\} \leq \rho_{1}\},
\label{ro8}
\end{equation}%
for some $\rho_{1}>0$.

By (\ref{majmiu1}), (\ref{majmiuL1}), (H1) and (H2), we have {\small
\begin{eqnarray*}
&&\Vert T_{1}\left( u,v\right) (t)\Vert _{\infty } \\
&=&\sup_{t\in \mathbb{R}}\left( \left\vert \int_{-\infty }^{t}\phi
^{-1}\left( \frac{\int_{-\infty }^{s}f(r,u(r),v(r),u^{\prime }(r),v^{\prime
}(r))dr}{a(s)}\right) ds\right. +A\right\vert \\
&\leq &\sup_{t\in \mathbb{R}}\int_{-\infty }^{t}\left\vert \phi ^{-1}\left(
\frac{\int_{-\infty }^{s}f(r,u(r),v(r),u^{\prime }(r),v^{\prime }(r))dr}{a(s)%
}\right) \right\vert ds+|A| \\
&\leq &\sup_{t\in \mathbb{R}}\int_{-\infty }^{t}\phi ^{-1}\left( \frac{%
\int_{-\infty }^{s}\left\vert f(r,u(r),v(r),u^{\prime }(r),v^{\prime
}(r))\right\vert dr}{a(s)}\right) ds+|A| \\
&\leq &\int_{-\infty }^{+\infty }\phi ^{-1}\left( \frac{\int_{-\infty
}^{+\infty }\vartheta _{\rho _{1}}(r)dr}{a(s)}\right) ds+|A|<+\infty ,
\end{eqnarray*}%
} and%
\begin{eqnarray*}
&&\Vert \left( T_{1}\left( u,v\right) \right) ^{\prime }(t)\Vert _{\infty }
\\
&=&\sup_{t\in \mathbb{R}}\left\vert \phi ^{-1}\left( \frac{\int_{-\infty
}^{t}f(r,u(r),v(r),u^{\prime }(r),v^{\prime }(r))dr}{a(t)}\right) \right\vert
\\
&\leq &\sup_{t\in \mathbb{R}}\phi ^{-1}\left( \frac{\int_{-\infty
}^{t}\left\vert f(r,u(r),v(r),u^{\prime }(r),v^{\prime }(r))\right\vert dr}{%
a(t)}\right) \\
&\leq &\sup_{t\in \mathbb{R}}\phi ^{-1}\left( \frac{\int_{-\infty }^{+\infty
}\vartheta _{\rho _{1}}(r)dr}{a(t)}\right) <+\infty .
\end{eqnarray*}

So, $\Vert T_{1}\left( u,v\right) (t)\Vert _{X}<+\infty $, that is, $T_{1}M$
is uniformly bounded on $X$.

By similar arguments, $T_{2}$ is uniformly bounded on $X$. Therefore $TM$ is
uniformly bounded on $X^2$.\medskip

\textbf{(iii)} $TM$\textit{\ is equicontinuous on }$X^2.$\newline

Let $t_{1},t_{2}\in \lbrack -K,K]\subseteq \mathbb{R}$ for some $K>0$, and
suppose, without loss of generality, that $t_{1}\leq t_{2}.$ Thus, by (\ref%
{majmiu1}), (\ref{majmiuL1}) and (H1),

\begin{eqnarray*}
&&\left\vert T_{1}\left( u,v\right) (t_{1})-T_{1}\left( u,v\right)
(t_{2})\right\vert \\
&=&\left\vert \int_{-\infty }^{t_{1}}\phi ^{-1}\left( \frac{\int_{-\infty
}^{s}f(r,u(r),v(r),u^{\prime }(r),v^{\prime }(r))dr}{a(s)}\right) ds\right.
\\
&&-\int_{-\infty }^{t_{2}}\phi ^{-1}\left( \frac{\int_{-\infty
}^{s}f(r,u(r),v(r),u^{\prime }(r),v^{\prime }(r))dr}{a(s)}\right) \left.
\frac{{}}{{}}ds\right\vert \\
&\leq &\int_{t_{1}}^{t_{2}}\phi ^{-1}\left( \frac{\int_{-\infty
}^{s}|f(r,u(r),v(r),u^{\prime }(r),v^{\prime }(r))dr|}{a(s)}\right) ds \\
&\leq &\int_{t_{1}}^{t_{2}}\phi ^{-1}\left( \frac{\int_{-\infty }^{+\infty
}\vartheta _{\rho _{1}}(r)dr}{a(s)}\right) ds\rightarrow 0,
\end{eqnarray*}%
uniformly for $(u,v)\in M$, as $t_{1}\rightarrow t_{2}$, and
\begin{eqnarray*}
&&\left\vert \left( T_{1}\left( u,v\right) \right) ^{\prime }(t_{1})-\left(
T_{1}\left( u,v\right) \right) ^{\prime }(t_{2})\right\vert \\
&=&\left\vert \phi ^{-1}\left( \frac{\int_{-\infty
}^{t_{1}}f(r,u(r),v(r),u^{\prime }(r),v^{\prime }(r))dr}{a(t_{1})}\right)
\right. \\
&&\left. -\phi ^{-1}\left( \frac{\int_{-\infty
}^{t_{2}}f(r,u(r),v(r),u^{\prime }(r),v^{\prime }(r))dr}{a(t_{2})}\right)
\right\vert \rightarrow 0,
\end{eqnarray*}%
uniformly for $(u,v)\in M$, as $t_{1}\rightarrow t_{2}$.

Therefore, $T_{1}M$ is equicontinuous on $X$. Analogously, it can be proved
that $T_{2}M$ is equicontinuous on $X$. So, $TM$ is equiconti\-nuous on $X^2$%
.\medskip

\textbf{(iv)} $TM$ \textit{is equiconvergent at $t=\pm \infty $}.\medskip

Let $(u,v)\in M$. For the operator $T_{1}$, we have, by (\ref{majmiu1}), (%
\ref{majmiuL1}) and (H1),
\begin{eqnarray*}
&&\left\vert T_{1}\left( u,v\right) (t)-\lim_{t\rightarrow -\infty
}T_{1}\left( u,v\right) (t)\right\vert \\
&=&\left\vert \int_{-\infty }^{t}\phi ^{-1}\left( \frac{\int_{-\infty
}^{s}f(r,u(r),v(r),u^{\prime }(r),v^{\prime }(r))dr}{a(s)}\right)
ds\right\vert \\
&\leq &\int_{-\infty }^{t}\phi ^{-1}\left( \frac{\int_{-\infty }^{+\infty
}\vartheta _{\rho _{1}}(r)dr}{a(s)}\right) ds\rightarrow 0,
\end{eqnarray*}%
uniformly in $(u,v)\in M$, as $t\rightarrow -\infty $, and,
\begin{eqnarray*}
&&\left\vert T_{1}\left( u,v\right) (t)-\lim_{t\rightarrow +\infty
}T_{1}\left( u,v\right) (t)\right\vert \\
&=&\left\vert \int_{-\infty }^{t}\phi ^{-1}\left( \frac{\int_{-\infty
}^{s}f(r,u(r),v(r),u^{\prime }(r),v^{\prime }(r))dr}{a(s)}\right) ds\right.
\\
&-&\int_{-\infty }^{+\infty }\phi ^{-1}\left( \frac{\int_{-\infty
}^{s}f(r,u(r),v(r),u^{\prime }(r),v^{\prime }(r))dr}{a(s)}\right) ds\left.
\frac{{}}{{}}\right\vert \\
&=&\left\vert \int_{t}^{+\infty }\phi ^{-1}\left( \frac{\int_{-\infty
}^{s}f(r,u(r),v(r),u^{\prime }(r),v^{\prime }(r))dr}{a(s)}\right)
ds\right\vert \\
&\leq &\int_{t}^{+\infty }\phi ^{-1}\left( \frac{\int_{-\infty }^{+\infty
}\vartheta _{\rho _{1}}(r)dr}{a(s)}\right) ds\rightarrow 0,
\end{eqnarray*}%
uniformly in $(u,v)\in M$, as $t\rightarrow +\infty $.

For the derivative it follows that,
\begin{eqnarray*}
&&\left\vert \left( T_{1}\left( u,v\right) \right) ^{\prime
}(t)-\lim_{t\rightarrow +\infty }\left( T_{1}\left( u,v\right) \right)
^{\prime }(t)\right\vert \\
&=&\left\vert \phi ^{-1}\left( \frac{\int_{-\infty
}^{t}f(r,u(r),v(r),u^{\prime }(r),v^{\prime }(r))dr}{a(t)}\right) \right. \\
&&\left. -\phi ^{-1}\left( \lim_{t\rightarrow +\infty }\frac{\int_{-\infty
}^{+\infty }f(r,u(r),v(r),u^{\prime }(r),v^{\prime }(r))dr}{a(t)}\right)
\right\vert \rightarrow 0,
\end{eqnarray*}%
uniformly in $(u,v)\in M$, as $t\rightarrow +\infty $, and%
\begin{eqnarray*}
&&\left\vert \left( T_{1}\left( u,v\right) \right) ^{\prime
}(t)-\lim_{t\rightarrow -\infty }\left( T_{1}\left( u,v\right) \right)
^{\prime }(t)\right\vert \\
&=&\left\vert \phi ^{-1}{\scriptstyle\left( \frac{\int_{-\infty
}^{t}f(r,u(r),v(r),u^{\prime }(r),v^{\prime }(r))dr}{a(t)}\right) }%
\right\vert \\
&\leq &\phi ^{-1}\left( \left\vert {\frac{\int_{-\infty
}^{t}f(r,u(r),v(r),u^{\prime }(r),v^{\prime }(r))dr}{a(t)}}\right\vert
\right) \\
&\leq &\phi ^{-1}\left( {\frac{\int_{-\infty }^{t}\vartheta _{\rho _{1}}(r)dr%
}{\left\vert a(t)\right\vert }}\right) \rightarrow 0
\end{eqnarray*}%
uniformly in $(u,v)\in M$, as $t\rightarrow -\infty $.

Therefore, $T_{1}M$ is equiconvergent at $\pm \infty $ and, following a
simi\-lar technique, we can prove that $T_{2}M$ is equiconvergent at $\pm
\infty $, too. So, $TM$ is equiconvergent at $\pm \infty $.

By Theorem \ref{F+H}, $TM$ is relatively compact. \medskip

\textbf{(v)} $T:X\rightarrow X$\textit{\ has a fixed point}$.\medskip $

In order to apply Schauder's fixed point theorem for operator $T\left(
u,v\right) ,$ we need to prove that $TD\subset D,$ for some closed, bounded
and convex $D\subset X^2.$

Consider
\begin{equation*}
D:=\left\{ (u,v)\in X^{2}:\left\Vert (u,v)\right\Vert _{X^{2}}\leq \rho
_{2}\right\} ,
\end{equation*}%
with $\rho _{2}>0$ such that
\begin{equation*}
\rho _{2}:=\max \left\{
\begin{array}{c}
\rho _{1},\;\int_{-\infty }^{+\infty }\phi ^{-1}\left( \frac{\int_{-\infty
}^{+\infty }\vartheta _{\rho _{2}}(r)dr}{a(s)}\right) ds+|A|, \\
\\
\int_{-\infty }^{+\infty }\psi ^{-1}\left( \frac{\int_{-\infty }^{+\infty
}\theta _{\rho _{2}}(r)dr}{b(s)}\right) ds+|B|, \\
\\
\sup_{t\in \mathbb{R}}\phi ^{-1}\left( \frac{\int_{-\infty }^{+\infty
}\vartheta _{\rho _{2}}(r)dr}{a(t)}\right) ,\;\sup_{t\in \mathbb{R}}\psi
^{-1}\left( \frac{\int_{-\infty }^{+\infty }\theta _{\rho _{2}}(r)dr}{b(t)}%
\right)%
\end{array}%
\right\}
\end{equation*}%
with $\rho _{1}$ given by (\ref{ro8}).

Following similar arguments as in \textbf{(ii)}, we have, for $(u,v)\in D,$%
\begin{eqnarray*}
\left\Vert T(u,v)\right\Vert _{X^{2}} &=&\left\Vert \left(
T_{1}(u,v),T_{2}\left( u,v\right) \right) \right\Vert _{X^{2}} \\
&=&\max \left\{ \left\Vert T_{1}\left( u,v\right) \right\Vert
_{X},\;\left\Vert T_{2}\left( u,v\right) \right\Vert _{X}\right\}  \\
&=&\max \left\{
\begin{array}{c}
\left\Vert T_{1}\left( u,v\right) \right\Vert _{\infty },\;\left\Vert \left(
T_{1}\left( u,v\right) \right) ^{\prime }\right\Vert _{\infty }, \\
\left\Vert T_{2}\left( u,v\right) \right\Vert _{\infty },\;\left\Vert \left(
T_{2}\left( u,v\right) \right) ^{\prime }\right\Vert _{\infty }%
\end{array}%
\right\} \leq \rho _{2},
\end{eqnarray*}%
and $TD\subset D$.

By Theorem \ref{schauder}, the operator $T\left( u,v\right) =\left(
T_{1}\left( u,v\right) ,T_{2}\left( u,v\right) \right) $ has a fixed point $%
(u,v)\in X^{2}$.

By standard arguments, it can be proved that this fixed point defines a pair
of heteroclinic or homoclinic solutions of problem (\ref{art8}), (\ref%
{cond8.1}).
\end{proof}

\begin{remark}
If
\begin{equation*}
\int_{-\infty }^{+\infty }\phi ^{-1}\left( \frac{\int_{-\infty
}^{s}f(r,u(r),v(r),u^{\prime }(r),v^{\prime }(r))dr}{a(s)}\right) ds=0
\end{equation*}%
and%
\begin{equation*}
\int_{-\infty }^{+\infty }\psi ^{-1}\left( \frac{\int_{-\infty
}^{s}h(r,u(r),v(r),u^{\prime }(r),v^{\prime }(r))dr}{b(s)}\right) ds=0.
\end{equation*}
the solutions $(u,\,v)\in X^{2}$ of problem (\ref{art8}), (\ref{cond8.1}),
will be a pair of homoclinic solutions.
\end{remark}

\section{Application to coupled systems of nonlinear 2-DOF model}

Generic nonlinear coupled systems of two degrees of freedom (2-DOF), are
especially important in Physics and Mechanics. For exam\-ple in \cite%
{Mikhlin+Perepelkin}, the authors use this type of system to investigate the
transient in a system containing a linear oscillator, li\-nearly coupled to
an essentially nonlinear attachment with a comparatively small mass. The
family of coupled non-linear systems of 2-DOF is used to study the global
bifurcations in the motion of an externally forced coupled nonlinear
oscillatory system or for the nonlinear vibration absorber subjected to
periodic excitation, see \cite{Malhotra}. Moreover, in \cite{Ariaratnam+Xie}%
, the authors deal with the stochastic moment stability of such systems.

Motivated by these works, in this section we consider an appli\-cation of
system (\ref{art8}), (\ref{cond8.1}), to a family of coupled non-linear
systems of 2-DOF model, given by the nonlinear coupled system (see \cite%
{Malhotra})

\begin{equation}
\left\{
\begin{array}{l}
\left((1+t^4) \left( q_{1}^{\prime }(t)\right) ^{3}\right) ^{\prime }=\frac{%
t^4}{(1+t^6)^2}\left[ 2\zeta \omega _{0}\left( q_{1}^{\prime }(t)\right)
^{3}+\omega _{0}^{2}q_{1}(t)+\gamma (\left( q_{1}(t)\right) ^{3}\right. \\
\;\;\;\;\;\;\;\;\;\;\;\;\;\;\;\;\;\;\;\;\;\;\;\;\;\;\;\;\;\;\;\;\;\;\;\;\;\;%
\left. -3d^{2}q_{1}(t)q_{2}(t))+\cos (t)\right] ,\text{ } \\
\\
\tau ^{2}\left((1+t^4)\left( q_{2}^{\prime }(t)\right) ^{3}\right) ^{\prime
}=\frac{t^4}{(1+t^6)^2} \left[ 2\zeta \omega _{0}\left( q_{2}^{\prime
}(t)\right) ^{3}+\omega _{0}^{2}q_{2}(t)+\gamma (d^{2}\left( q_{2}(t)\right)
^{3}\right. \\
\;\;\;\;\;\;\;\;\;\;\;\;\;\;\;\;\;\;\;\;\;\;\;\;\;\;\;\;\;\;\;\;\;\;\;\;\;\;%
\;\;\left. -3\left( q_{1}(t)\right) ^{2}q_{2}(t))\right] ,\text{ }t\in
\mathbb{R},%
\end{array}%
\right.  \label{EQAp1}
\end{equation}%
where

\begin{itemize}
\item $q_{1}(t)$ and $q_{2}(t)$ represent the generalized coordinates;

\item $d,\;\tau ,\;\gamma $ are positive constant coefficients which depend
on the characteristics of the physical or mechanical system under
consideration;

\item $\cos(t)$ is related to the type of excitation of the system under%
\newline
consideration;

\item $\zeta ,\omega _{0}$, are the damping coefficient and the frequency,
respectively.
\end{itemize}

As the asymptotic conditions we consider%
\begin{equation}
q_{1}(-\infty )=A,\;\,q_{1}^{\prime }(+\infty )=0,\text{ }q_{2}(-\infty )=B,%
\text{ }q_{2}^{\prime }(+\infty )=0,  \label{BCap}
\end{equation}%
with $A,B\in \mathbb{R},$ and, moreover, assume that the real coefficients $%
\zeta $, $\omega _{0}$, $\gamma $, $d$, $r$ are such that the integrals
\begin{equation}
\int_{-\infty }^{+\infty }\left( \sqrt[3]{\frac{\int_{-\infty }^{s}\frac{%
r^{4}}{(1+r^{6})^{2}}\left[
\begin{array}{c}
2\zeta \omega _{0}\left( q_{1}^{\prime }(r)\right) ^{3}+\omega
_{0}^{2}q_{1}(r)+\gamma (\left( q_{1}(r)\right) ^{3} \\
-3d^{2}q_{1}(r)q_{2}(r))+\cos (r)%
\end{array}%
\right] dr}{1+s^{4}}}\right) ds  \label{BAap}
\end{equation}%
and%
\begin{equation}
\int_{-\infty }^{+\infty }\left( \sqrt[3]{\frac{\int_{-\infty }^{s}\frac{%
r^{4}}{\tau ^{2}(1+r^{6})^{2}}\left[
\begin{array}{c}
2\zeta \omega _{0}\left( q_{2}^{\prime }(r)\right) ^{3}+\omega
_{0}^{2}q_{2}(r)+\gamma (d^{2}\left( q_{2}(r)\right) ^{3} \\
-3\left( q_{1}(r)\right) ^{2}q_{2}(r))%
\end{array}%
\right] dr}{1+s^{4}}}\right) ds.  \label{CDap}
\end{equation}%
are finite.

It is clear that (\ref{EQAp1}) is a particular case of (\ref{art8}) with:
\begin{equation*}
\phi (z)=\psi (z)=z^{3},\;\,a(t)=b(t)=1+t^{4},
\end{equation*}%
$f,h:\mathbb{R}^{5}\rightarrow \mathbb{R}$ are $L^{1}$-Carath\'{e}odory
functions where
\begin{eqnarray*}
f(t,x,y,z,w) &=&\frac{t^{4}}{(1+t^{6})^{2}}\left( 2\zeta \omega
_{0}z^{3}+\omega _{0}^{2}x+\gamma x^{3}-3d^{2}xy+\cos (t)\right)  \\
&\leq &\frac{t^{4}}{\left( t^{6}+1\right) ^{2}}\left( 2\left\vert \zeta
\omega _{0}\right\vert \rho ^{3}+\omega _{0}^{2}\rho +\gamma \rho
^{3}+3d^{2}\rho ^{2}+1\right)  \\
:= &&\delta _{\rho }(t)
\end{eqnarray*}%
\begin{eqnarray*}
h(t,x,y,z,w) &=&\frac{t^{4}}{\tau ^{2}(1+t^{6})^{2}}\left( 2\zeta \omega
_{0}w^{3}+\omega _{0}^{2}y+\gamma d^{2}y^{3}-3x^{2}y\right)  \\
&\leq &\frac{t^{4}}{\tau ^{2}(t^{6}+1)^{2}}\left( 2\left\vert \zeta \omega
_{0}\right\vert \rho ^{3}+\omega _{0}^{2}\rho +\gamma d^{2}\rho ^{3}+3\rho
^{3}\right)  \\
:= &&\varepsilon _{\rho }(t)
\end{eqnarray*}%
where $\delta _{\rho }(t)$ and $\varepsilon _{\rho }(t)$ are functions in $%
L^{1}(\mathbb{R}),$ for $\rho >0$ such that%
\begin{equation}
\rho :=\max \left\{ \left\vert x\right\vert ,\left\vert y\right\vert
,\left\vert z\right\vert ,\left\vert w\right\vert \right\} .  \label{roapl}
\end{equation}

Moreover, conditions (H1) and (H2) hold as,

\begin{itemize}
\item $\phi(\mathbb{R})=\psi(w)=\mathbb{R}$ and $\phi(0)=\psi(0)=0$;

\item $|\phi ^{-1}(z)|=|\sqrt[3]{z}|=\phi ^{-1}(|z|)=\sqrt[3]{|z|}$ and $%
|\psi ^{-1}(w)|=|\sqrt[3]{w}|=\psi ^{-1}(|w|)=\sqrt[3]{|w|}$;

\item $\underset{t\rightarrow \pm \infty }{\lim }\frac{1}{a(t)} =\underset{%
t\rightarrow \pm \infty }{\lim }\frac{1}{1+t^4} =\underset{t\rightarrow \pm
\infty }{\lim }\frac{1}{b(t)}=0$.
\end{itemize}

For $\rho >0$ such that

\begin{eqnarray}
&&\int_{-\infty }^{+\infty }\phi ^{-1}\left( \frac{\int_{-\infty }^{+\infty
}\delta _{\rho }(r)dr}{a(s)}\right) ds  \notag \\
&=&\int_{-\infty }^{+\infty }\left( \sqrt[3]{\frac{\int_{-\infty }^{+\infty
}r^{4}\frac{\left( 2\left\vert \zeta \omega _{0}\right\vert \rho ^{3}+\omega
_{0}^{2}\rho +\gamma \rho ^{3}+3d^{2}\rho ^{2}+1\right) }{\left(
1+r^{6}\right) ^{2}}dr}{1+s^{4}}}\right) ds<\rho \;\;\,\;\;\;\;\;\;\;\,
\label{cap8rho11}
\end{eqnarray}%
and
\begin{eqnarray}
&&\int_{-\infty }^{+\infty }\psi ^{-1}\left( \frac{\int_{-\infty }^{+\infty
}\varepsilon _{\rho }(r)dr}{b(s)}\right) ds  \notag \\
&=&\int_{-\infty }^{+\infty }\left( \sqrt[3]{\frac{\int_{-\infty }^{+\infty
}r^{4}\frac{\left( 2\left\vert \zeta \omega _{0}\right\vert \rho ^{3}+\omega
_{0}^{2}\rho +\gamma d^{2}\rho ^{3}+3\rho ^{3}\right) }{\tau ^{2}\left(
1+r^{6}\right) ^{2}}dr}{1+s^{4}}}\right) ds<\rho ,\;\;\;\;\;\;\,\,
\label{cap8rho12}
\end{eqnarray}%
by Theorem \ref{main8}, the system (\ref{EQAp1}) together with the
asymptotic conditions (\ref{BCap}), has at least a pair $(q_{1},q_{2})\in
X^{2}$ of heteroclinic solutions, since the integrals (\ref{BAap}) and (\ref%
{CDap}) are finite. As example, in particular, for
\begin{equation*}
|\zeta |=\frac{1}{2\sqrt{1000}},\;|\omega _{0}|=\frac{1}{\sqrt{1000}}%
,\;\gamma =\frac{1}{1000},\;d^{2}=\frac{1}{3000},\;\tau=23
\end{equation*}%
the conditions (\ref{cap8rho11}) and (\ref{cap8rho12}) hold for $\rho >
6.3542$.

For the values of the above parameters, $A=10$ and $B=8$, the heteroclinics
solutions  $q_{1}$ and $q_{2}$ have the graphs given in Figure 1. In Figure
2 we present the real shape of the $q_{1}$ trajectory, which is not detailed in
Figure 1 due to the scale range.

\begin{figure}[h]
\label{Figura1} \center
\includegraphics[width=6cm]{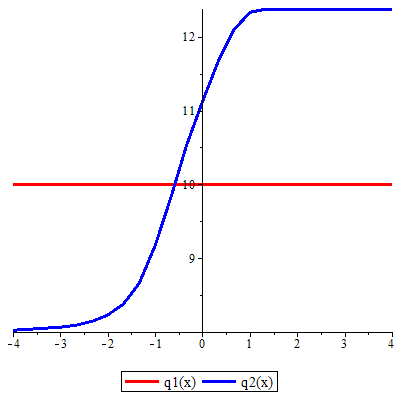}
\caption{A=10, B=8 }
\end{figure}
\begin{figure}[h]
\label{Figura2} \center
\includegraphics[width=6cm]{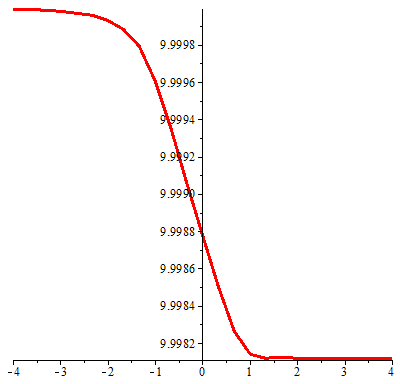}
\caption{A=10 }
\end{figure}

Remark that, if the integrals (\ref{BAap}) and (\ref{CDap}) are null, then
he system (\ref{EQAp1}) has a pair of homoclinic solutions $(q_{1},q_{2})\in
X^{2}.$

\end{document}